\documentclass[letterpaper]{amsart}
\pdfoutput=1
\newtheorem{theorem}{Theorem}[section]
\newtheorem{lemma}[theorem]{Lemma}
\newtheorem{proposition}[theorem]{Proposition}
\newtheorem{corollary}[theorem]{Corollary}

\theoremstyle{definition}
\newtheorem{definition}[theorem]{Definition}

\theoremstyle{remark}
\newtheorem{remark}[theorem]{Remark}
\usepackage{hyperref}
\usepackage{graphicx}
\usepackage{color}
\usepackage{amsmath}
\usepackage{amsfonts}
\usepackage{amssymb}
\usepackage{epsfig}
\usepackage{rotating}
\usepackage{graphics}
\newcommand{\be}{\begin{equation}}

\newcommand{\ee}{\end{equation}}

\newcommand{\Om}{\Omega}

\newcommand{\om}{\omega}

\newcommand{\la}{\lambda}

\newcommand{\dz}{\wedge}

\newcommand{\ba}{\begin{array}}

\newcommand{\ea}{\end{array}}

\newcommand{\beq}{\begin{eqnarray}}

\newcommand{\eeq}{\end{eqnarray}}

\newtheorem{lm}{lemma}

\newtheorem{thee}{theorem}

\newtheorem{proo}{proposition}

\newtheorem{co}{corollary}

\newtheorem{rem}{remark}

\newtheorem{deff}{definition}

\newcommand{\bd}{\begin{deff}}

\newcommand{\ed}{\end{deff}}

\newcommand{\bl}{\begin{lm}}

\newcommand{\el}{\end{lm}}

\newcommand{\bp}{\begin{proo}}

\newcommand{\ep}{\end{proo}}

\newcommand{\bt}{\begin{thee}}

\newcommand{\et}{\end{thee}}

\newcommand{\bc}{\begin{co}}

\newcommand{\ec}{\end{co}}

\newcommand{\brm}{\begin{rem}}

\newcommand{\erm}{\end{rem}}

\newcommand{\der}{{\rm d}}

\usepackage{t1enc}
\def\frak{\mathfrak}

\newcommand{\newc}{\newcommand}

\let\ccdot\cdot
\def\cdot{\hbox to 2.5pt{\hss$\ccdot$\hss}}

\newc{\aR}{\mbox{\boldmath{$ R$}}}
\newc{\aS}{\mbox{\boldmath{$ S$}}}
\newc{\aT}{\mbox{\boldmath{$ T$}}}
\newc{\aW}{\mbox{\boldmath{$ W$}}}

\newc{\aK}{\mbox{\boldmath{$ K$}}}
\newc{\aL}{\mbox{\boldmath{$ L$}}}

\usepackage{amssymb}
\usepackage{amscd}

\newcommand{\hook}{\raisebox{-0.35ex}{\makebox[0.6em][r]
{\scriptsize $-$}}\hspace{-0.15em}\raisebox{0.25ex}{\makebox[0.4em][l]{\tiny
 $|$}}}

\newcommand{\bma}{\begin{pmatrix}}
\newcommand{\ema}{\end{pmatrix}}

\newc{\obstrn}[2]{B^{#1}_{#2}}

\newcommand{\rpl}                         
{\mbox{$
\begin{picture}(12.7,8)(-.5,-1)
\put(0,0.2){$+$}
\put(4.2,2.8){\oval(8,8)[r]}
\end{picture}$}}

\newcommand{\lpl}                         
{\mbox{$
\begin{picture}(12.7,8)(-.5,-1)
\put(2,0.2){$+$}
\put(6.2,2.8){\oval(8,8)[l]}
\end{picture}$}}

\usepackage{ifthen}

\newc{\tensor}[1]{#1}
\newc{\Mvariable}[1]{\mbox{#1}}
\newc{\down}[1]{{}_{#1}}
\newc{\up}[1]{{}^{#1}}

\newc{\JulyStrut}{\rule{0mm}{6mm}}
\newc{\midtenPan}{\mbox{\sf S}}
\newc{\midten}{\mbox{\sf T}}
\newc{\midtenEi}{\mbox{\sf U}}
\newc{\ATen}{\mbox{\sf E}}
\newc{\BTen}{\mbox{\sf F}}
\newc{\CTen}{\mbox{\sf G}}

\def\sideremark#1{\ifvmode\leavevmode\fi\vadjust{\vbox to0pt{\vss
 \hbox to 0pt{\hskip\hsize\hskip1em
 \vbox{\hsize3cm\tiny\raggedright\pretolerance10000
 \noindent #1\hfill}\hss}\vbox to8pt{\vfil}\vss}}}%

\newcommand{\bgw}{{\textstyle \bigwedge}}

\numberwithin{equation}{section}

\newcounter{romenumi}
\newcommand{\labelromenumi}{(\roman{romenumi})}

\newcommand{\bbT}{\mathbb{T}}

\begin{document}
\title{Twistor space for rolling bodies}
\dedicatory{Dedicated to Mike Eastwood on the occasion of his 60th birthday}
\vskip 1.truecm
\author{Daniel An} \address{SUNY Maritime College 6 Pennyfield Avenue, Throggs Neck, New York 10465}
\email{dan@sunymaritime.edu}
\author{Pawe\l~ Nurowski} \address{Centrum Fizyki Teoretycznej,
Polska Akademia Nauk, Al. Lotnik\'ow 32/46, 02-668 Warszawa, Poland}
\email{nurowski@cft.edu.pl}
\thanks{This research was partially supported by the Polish Ministry of
  Research and Higher Education under grants NN201 607540 and NN202
  104838}

\date{\today}
\begin{abstract}
On a natural circle bundle $\bbT(M)$ over a 4-dimensional manifold $M$ equipped with a split signature metric $g$, whose fibers are real totally null selfdual 2-planes, we consider a tautological rank 2 distribution $\mathcal D$ obtained by lifting each totally null plane horizontally to its point in the fiber. Over the open set where $g$ is not antiselfdual, the distribution $\mathcal D$ is (2,3,5) in $\bbT (M)$. We show that if $M$ is a Cartesian product of two Riemann surfaces $(\Sigma_1,g_1)$ and $(\Sigma_2,g_2)$, and if $g=g_1\oplus(- g_2)$, then the circle bundle $\bbT(\Sigma_1\times\Sigma_2)$ is just the configuration space for the physical system of two surfaces $\Sigma_1$ and $\Sigma_2$ rolling on each other. The condition for the two surfaces to roll on each other `without slipping or twisting' identifies the restricted velocity space for such a system with the tautological distribution $\mathcal D$ on $\bbT(\Sigma_1\times\Sigma_2)$. We call $\bbT(\Sigma_1\times\Sigma_2)$ the \emph{twistor space}, and $\mathcal D$ the \emph{twistor distribution} for the rolling surfaces. Among others we address the following question: "For which pairs of surfaces does the restricted velocity distribution (which we identify with the twistor distribution $\mathcal D$) have the simple Lie group $G_2$ as the group of its symmetries?" Apart from the well known situation when the surfaces $\Sigma_1$ and $\Sigma_2$ have constant curvatures whose ratio is 1:9, we unexpectedly find \emph{three} different types of surfaces that when rolling `without slipping or twisting' on a \emph{plane}, have $\mathcal D$ with the symmetry group $G_2$. Although we have found the differential equations for the curvatures of $\Sigma_1$ and $\Sigma_2$ that gives $\mathcal D$ with $G_2$ symmetry, we are unable to solve them in full generality so far.
\end{abstract}
\maketitle
\vspace{-1truecm}
\tableofcontents
\newcommand{\bbS}{\mathbb{S}}
\newcommand{\bbR}{\mathbb{R}}
\newcommand{\sog}{\mathbf{SO}}
\newcommand{\slg}{\mathbf{SL}}
\newcommand{\og}{\mathbf{O}}
\newcommand{\soa}{\frak{so}}
\newcommand{\sla}{\frak{sl}}
\newcommand{\sua}{\frak{su}}
\newcommand{\dr}{\mathrm{d}}
\newcommand{\sug}{\mathbf{SU}}
\newcommand{\gat}{\tilde{\gamma}}
\newcommand{\Gat}{\tilde{\Gamma}}
\newcommand{\thet}{\tilde{\theta}}
\newcommand{\Thet}{\tilde{T}}
\newcommand{\rt}{\tilde{r}}
\newcommand{\st}{\sqrt{3}}
\newcommand{\kat}{\tilde{\kappa}}
\newcommand{\kz}{{K^{{~}^{\hskip-3.1mm\circ}}}}
\newcommand{\bv}{{\bf v}}
\newcommand{\di}{{\rm div}}
\newcommand{\curl}{{\rm curl}}
\newcommand{\cs}{(M,{\rm T}^{1,0})}
\newcommand{\tn}{{\mathcal N}}
\section{Introduction}
Bryant and Hsu \cite{brya}, pp. 456-458, gave the following description 
of the configuration space of two solids rolling on each other `without slipping or twisting': 

The two solids are represented by two surfaces $\Sigma_1$ and $\Sigma_2$, equipped with the respective Riemannian metrics $g_1$ and $g_2$. The configuration space for the physical system is parmetrized by points $x$ on the first surface, points $\hat{x}$ on the second surface (these are just the points of contact of the two surfaces), and a rotation $A$ identifying the tangent space to $\Sigma_1$ at $x$ with the tangent space to $\Sigma_2$ at $\hat{x}$. This makes the configuration space a circle fiber bundle $\bbS^1\hookrightarrow C(\Sigma_1,\Sigma_2)\to \Sigma_1\times\Sigma_2$ over the Cartesian product $\Sigma_1\times\Sigma_2$ of the two surfaces,
$$C(\Sigma_1,\Sigma_2)=\{(x,\hat{x},A)~|~A:{\rm T}_x\Sigma_1\to{\rm T}_{\hat{x}}\Sigma_2,~A\in\sog(2)\cong\bbS^1\},$$
with the projection $\pi(x,\hat{x},A)=(x,\hat{x})$. 

In this realization of the configuration space, the movement of the two surfaces is represented by curves $\gamma(t)=(x(t),\hat{x}(t),A(t))$ in $C(\Sigma_1,\Sigma_2)$. The unconstrained velocity space at a point $p$ consists of all vectors of the form $\dot{\gamma}(t)_{|t=0}=(\dot{x}(t),\dot{\hat{x}}(t),$ $\dot{A}(t))_{|t=0}$, where $\gamma(t)$ stands for all smooth curves in $C(\Sigma_1,\Sigma_2)$ such that $\gamma(0)=p$.  

The `no slipping and no twisting' conditions constrain the velocity space, reducing its dimension at each point from five to two. This reduction is obtained by first imposing a condition for the absence of `linear slipping'. This can be formalized as follows. If $\gamma(t)=(x(t),\hat{x}(t),A(t))$ is an admissible motion, then the lack of linear slipping means that:
$$A(t)\dot{x}(t)=\dot{\hat{x}}(t).$$
This produces a drop in the dimension of the velocity space at each point by two, from five to three. The condition of no `twisting' reduces this dimension to two. We impose it now. It means that the admissible motions $\gamma(t)=(x(t),\hat{x}(t),A(t))$ must have the following geometric property: for every vector field $v(t)$ which is parallel along $x(t)$ the $A(t)$ transformed vector field  $\hat{v}(t)$ must be a vector field parallel along $\hat{x}(t)$, i.e.
$$\stackrel{1}{\nabla}_{x(t)} v(t)=0 ~~{\rm and } ~~A(t)v(t)=\hat{v}(t) ~~{\rm implies} ~~\stackrel{2}{\nabla}_{\hat{x}(t)} \hat{v}(t)=0,$$   
 where $\stackrel{i}{\nabla}$ is the Levi-Civita connection for the surface $(\Sigma_i, g_i )$.
 
To be more explicit, we now follow \cite{agr}. We take $(e_1(x),e_2(x))$ as an orthonormal frame in $\Sigma_1$ and $(e_3(\hat{x}),e_4(\hat{x}))$ as an orthonormal frame in $\Sigma_2$. To simplify the notation, from now on we will omit the dependencies of $x$ and $\hat{x}$ in the expressions involving these basis vectors. 

The most general forms of the commutators for $(e_1,e_2)$ and $(e_3,e_4)$ are:
\be
[e_1,e_2]=a_1e_1+a_2e_2,\quad\quad [e_3,e_4]=a_3 e_3+a_4 e_4,\label{fr}\ee
with $a_1=a_1(x), a_2=a_2(x)$ functions on $\Sigma_1$ and $a_3=a_3(\hat{x}), a_4=a_4(\hat{x})$ functions on $\Sigma_2$. We extend the coframes $(e_1,e_2)$ and $(e_3,e_4)$ to a coframe $(e_1,e_2,e_3,e_4)$ on $\Sigma_1\times\Sigma_2$. This is done by requiring that the extended frame $(e_1,e_2,e_3,e_4)$ satisfies (\ref{fr}), with the functions $a_1, a_2$ being constant along $e_3$ and $e_4$, and the functions $a_3$ and $a_4$ being constant along $e_1$ and $e_2$. The next requirement, that uniquely defines the extension, is that all commutators of $(e_1,e_2,e_3,e_4)$ other than those given by the relations (\ref{fr}) vanish on $\Sigma_1\times\Sigma_2$ . Parametrizing the rotation matrices $A$ by the angle of rotation $\phi$, 
$$A_\phi=\bma \cos \phi &-\sin \phi \\\sin \phi &\cos \phi \ema,$$
we further extend the frame $(e_1,e_2,e_3,e_4)$ from $\Sigma_1\times\Sigma_2$ to $C(\Sigma_1,\Sigma_2)$ by the requirement that the resulting vector fields $(e_1,e_2,e_3,e_4)$ on $C(\Sigma_1,\Sigma_2)$ are constant when Lie dragged along the fibers:
$${\mathcal L}_{\partial_\phi}e_i\equiv 0,\quad\quad i=1,2,3,4.$$
This defines a coframe $(e_1,e_2,e_3,e_4,\partial_\phi)$ in $C(\Sigma_1,\Sigma_2)$. Now, it follows from \cite{agr} that the velocity space of admissible motions constrained by the `no slipping and no twisting' conditions is, at every point, spanned by:
\be
\begin{aligned}
&\tilde{X}_1=e_1+\cos \phi e_3+\sin \phi e_4+(-a_1+a_3\cos \phi +a_4\sin \phi )\partial_\phi\\
&\tilde{X}_2=e_2-\sin \phi e_3+\cos \phi e_4+(-a_2-a_3\sin \phi +a_4\cos \phi )\partial_\phi.
\end{aligned}\label{dist}
\ee 
We summarize the above considerations in the following proposition.
\begin{proposition}
The configuration space for the physical system of two surfaces rolling on each other `without slipping or twisting' is a circle bundle 
$\bbS^1\hookrightarrow C(\Sigma_1,\Sigma_2)$ $\to \Sigma_1\times\Sigma_2.$
The space of admissible velocities for the system is a 2-dimensional distribution ${\mathcal D}_v$ in $C(\Sigma_1,\Sigma_2)$. In coordinates $(x,\hat{x},\phi)$ on $C(\Sigma_1,\Sigma_2)$, where $x$ and $\hat{x}$ denote the respective points on $\Sigma_1$ and $\Sigma_2$, and where $\phi$ is the angle of rotation corresponding to the map $A_\phi$, the distribution ${\mathcal D}_v$ is spanned by the vector fields $\tilde{X}_1$ and $\tilde{X}_2$ given by (\ref{dist}). 
\end{proposition}
 
\begin{remark}\label{homo}
Note that if we simultaneously rescale the metric of our two Riemann surfaces by the same constant, i.e. $(g_1,g_2)\to (s^2 g_1,s^2 g_2)$ with $s={\rm const}\neq 0$, then $e_i\to s^{-1}e_i$ and $a_i\to s^{-1}a_i$, $i=1,2,3,4$. This transformation merely rescales the vector fields $\hat{X}_1$, $\hat{X}_2$ as $\hat{X}_1\to s^{-1}\hat{X}_1$ and $\hat{X}_2\to s^{-1}\hat{X}_2$. Thus the distribution ${\mathcal D}_v$ does not change when the two rolling surfaces are scaled by the same constant factor. This reflects an obvious fact that the local symmetry of two surfaces rolling on each other `without slipping or twisting' should only depend on their relative size respect to one another.   
\end{remark}

We now present a simple observation that is crucial for the rest of the paper:

\begin{proposition}\label{mike}
Every point of the configuration space $C(\Sigma_1,\Sigma_2)$ of the system of two rolling surfaces `without slipping or twisting' defines a 2-plane, which is totally null in the standard split signature metric in $\bbR^4=\bbR^{(2,2)}$.
\end{proposition}
\begin{proof}
Given a point $(x,\hat{x},\phi)$ in $C(\Sigma_1,\Sigma_2)$ we consider the graph
$$\{(a,b,a\cos\phi-b\sin\phi,a\sin\phi+b\cos\phi)~|~a,b\in \bbR^2\}\subset\bbR^4,$$
of the map $A_\phi:{\rm T}_x\Sigma_1\to{\rm T}_{\hat{x}}\Sigma_2$. This gives a plane $$N(x,\hat{x},\phi)={\rm Span}(X_1,X_2)$$
in $\bbR^4$ spanned by the vectors $$X_1=(1,0,\cos\phi,\sin\phi)\quad {\rm and}\quad X_2=(0,1,-\sin\phi,\cos\phi).$$
Due to to the orthogonality of $A$, the plane $N(x,\hat{x},\phi)$ is totally null in the standard split signature metric $y_1^2+y_2^2-y_3^2-y_4^2$ in $\bbR^4=\bbR^{(2,2)}$. 
\end{proof}
This proposition suggests that we consider the space $\bbT(M)$ of real totally null planes over a 4-dimensional manifold $M=\Sigma_1\times\Sigma_2$ equipped with the metric $g=g_1\oplus (-g_2)$ and identify the points of the configuration space $C(\Sigma_1,\Sigma_2)$ for the two rolling surfaces with the points of $\bbT(M)$. To make this suggestion into a precise identification we now discuss the geometry of the space $\bbT(M)$. Because of possible applications other than the kinematics of the rolling surfaces, we will consider $\bbT(M)$ over \emph{general} split signature metric 4-manifolds $M$, \emph{not} assuming from the very beginning that $M$ is a product of two surfaces.
\section{Twistor space}
\subsection{Null planes in $\bbR^{(2,2)}$}
Consider the 4-dimensional vector space $V=\bbR^4$. Denote by $(e_1,e_2,e_3,e_4)$ the standard basis in it, $e_1=(1,0,0,0)$, $e_2=(0,1,0,0)$, $e_3=(0,0,1,0)$ and $e_4=(0,0,0,1)$. Then every vector $y\in V$ is $y=y_1e_1+y_2e_2+y_3e_3+y_4 e_4$. 

We now endow $V$ with the standard split signature metric $g$, by setting $g(y,y)=y_1^2+y_2^2-y_3^2-y_4^2$ for each $y\in V$. We also choose an orientation in $V$. This additionally equips $V$ with the Hodge star operator $*$ which, in particular, is an automorphism of the space $\bgw^2V$ of bivectors. In the basis $e_i$ this automorphism is given by 
\be
\begin{aligned}
&*(e_1\dz e_2)=e_3\dz e_4,\quad &*(e_3\dz e_4)=e_1\dz e_2\\
&*(e_1\dz e_3)=e_2\dz e_4,\quad &*(e_2\dz e_4)=e_1\dz e_3\\
&*(e_1\dz e_4)=-e_2\dz e_3,\quad &*(e_2\dz e_3)=-e_1\dz e_4.
\end{aligned}\label{ob}
\ee
One easily checks that the map $*:\bgw^2 V\to \bgw^2 V$ squares to the identity, $*^2={\rm Id}$. It has two eigenvalues $+1$ and $-1$, and splits $\bgw^2 V$ onto a direct sum of the corresponding eigenspaces $\bgw^2V=V_+\oplus V_-$. 
Bivectors from $V_+$ are called selfdual, and bivectors from $V_-$ are called antiselfdual.

In $V$ we have two kinds of real totally null planes. An example of the planes of the first kind is 
\be
N_+={\rm Span}(e_1+e_3,e_2+e_4)\label{n+}\ee and an example of the planes of the second kind is 
\be N_-={\rm Span}(e_1+e_3,e_2-e_4).\label{n-}\ee The difference between them is clearly visible in terms of their corresponding bivectors: 

Let $N={\rm Span}(n_1,n_2)$ be a general real totally null plane in $V$. This means that $n_1,n_2\in V$, $g(n_1,n_1)=g(n_1,n_2)=g(n_2,n_2)=0$ and $n_1\dz n_2\neq 0$. Every such $N$ defines a line $\bbR L(N)$ in $\bgw^2V$ represented by $L(N)=n_1\dz n_2$. One can show that the condition that $N$ is totally null forces $L(N)$ to be an eigenvector of $*$. Thus $L(N)$ is either selfdual or antiselfdual, and we use this property of $L(N)$ to call the corresponding $N$ selfdual, or antiselfdual respectively. In this sense, our $N_+$ above is selfdual, and $N_-$ is antiselfdual.

The identity component $\sog_0(2,2)$ of the orthogonal group $\sog(2,2)$ acts on totally null planes via:
$$hN={\rm Span}(hn_1,hn_2), \quad{\rm where}\quad h\in \sog_0(2,2),\quad N={\rm Span}(n_1,n_2),$$
where $hn_1$ denotes the usual action of $\sog_0(2,2)$ on the vector $n_1$ in $\bbR^4=\bbR^{(2,2)}$. 

This action has two orbits ${\mathcal O}_+$ and ${\mathcal O}_-$ given by:
$${\mathcal O}_\pm=\{hN_\pm~|~h\in\sog_0(2,2)\},$$
where $N_\pm$ is given by (\ref{n+}) and (\ref{n-}), respectively. Thus each orbit consists of all the totally null planes of a given selfduality. Both of them are diffeomorphic to a circle $\bbS^1$. We summarize considerations of this section in the following (well known) proposition.
\begin{proposition}
The space $\mathcal O$ of totally null planes in $V=\bbR^4$ equipped with the split signature metric is a disjoint union, ${\mathcal O}={\mathcal O}_+\bigsqcup{\mathcal O}_-$ of the spaces ${\mathcal O}_\pm$ of respectively selfdual and antiselfdual totally null planes. Each of the spaces ${\mathcal O}_\pm$ is diffeomorphic to a circle, ${\mathcal O}_\pm\cong \bbS^1$. In the orthonormal basis (\ref{ob}) the orbit ${\mathcal O}_+$ may be parametrized by $\phi\in [0,2\pi[$, so that $N_+(\phi)\in{\mathcal O}_+$ iff
$$
N_+(\phi)={\rm Span}(e_1+\cos\phi e_3+\sin\phi e_4,e_2+\sin\phi e_3-\cos\phi e_4).$$
Similarly the orbit ${\mathcal O}_-$ consists of points 
$$N_-(\phi)={\rm Span}(e_1+\cos\phi e_3+\sin\phi e_4,e_2-\sin\phi e_3+\cos\phi e_4).$$
The corresponding lines of bivectors are:
$$\begin{aligned}
&\bbR L(N_+(\phi))=\\&{\rm Span}\Big(e_1\dz e_2+e_3\dz e_4-\sin\phi(e_1\dz e_3+e_2\dz e_4)+\cos\phi(e_1\dz e_4-e_2\dz e_3)\Big)\end{aligned}$$
and
$$\begin{aligned}
&\bbR L(N_-(\phi))=\\&{\rm Span}\Big(e_1\dz e_2-e_3\dz e_4+\sin\phi(e_1\dz e_3-e_2\dz e_4)-\cos\phi(e_1\dz e_4+e_2\dz e_3)\Big).\end{aligned}$$
\end{proposition}  
\subsection{Null planes on a manifold}
We now consider a 4-dimensional real oriented manifold $M$ equipped with a split signature metric $g$. We use an orthonormal coframe $(\sigma^1,\sigma^2,\sigma^3,\sigma^4)$ in which the metric looks like
\be
g=g_{ij}\sigma^i\sigma^j=(\sigma^1)^2+(\sigma^2)^2-(\sigma^3)^2-(\sigma^4)^2,\label{met}\ee
with its dual frame of vector fields $(e_1,e_2,e_3,e_4)$ on $M$. We then have $e_i\hook\sigma^j=\delta_i^{~j}$. At every point $y\in M$, we have a circle 
$${\mathcal O}_+(y)=\{N_+(y,\phi)={\rm Span}(e_1+\cos\phi e_3+\sin\phi e_4,e_2+\sin\phi e_3-\cos\phi e_4)~|~\phi\in[0,2\pi[\}$$
 of real totally null planes $N_+(\phi)$. A disjoint union $\bbT(M)$ of these circles, $$\bbT(M)=\bigcup_{y\in M}{\mathcal O}_+(y),$$
as $y$ runs through all the points of $M$, is a circle bundle
$$\bbS^1\hookrightarrow\bbT(M)\stackrel{\pi}{\to}M,$$
with the projection
$$\pi(y,N_+(y,\phi))=y$$
and fibers
$$\pi^{-1}(y)={\mathcal O}_+(y).$$
One sees that the points $(y,N_+(y,\phi))$ are uniquely parametrized by $(y,\phi)$, $y\in M$, $\phi\in [0,2\pi[$. We will use this parametrization of $\bbT(M)$ in the following.

\begin{definition}
Given a 4-dimensional oriented manifold $M$ equipped with a split signature metric $g$ its natural circle bundle $\bbS^1\hookrightarrow\bbT(M)\stackrel{\pi}{\to}M$ defined above is called a \emph{twistor (circle) bundle}. 
\end{definition}
Twistor bundle $\bbT(M)$ has an additional structure induced by the Levi-Civita connection from $M$ (see \cite{AHS,Pen} for more details, and e.g. \cite{spar,nurtwis} for the formulation in terms of totally null planes). 
\begin{proposition}\label{ho}
The tangent bundle ${\rm T}\bbT(M)$ to the twistor circle bundle $\bbS^1\hookrightarrow\bbT(M)\to M$ of a 4-dimensional manifold $M$ equipped with a split signature metric $g$ naturally splits into vertical $\mathcal V$ and horizontal $\mathcal H$ parts
$${\rm T}\bbT(M)={\mathcal V}\oplus\mathcal H.$$
This equips $\bbT(M)$ with a canonical rank two distribution $\mathcal D$ whose 2-plane at each point $(y,\phi)\in\bbT(M)$ is given by the \emph{horizontal lift} of a totally null plane $N_+(y,\phi)$ from $y\in M$ to $(y,\phi)\in\bbT(M)$.
\end{proposition}
\begin{proof}
Of course, the vertical space $\mathcal V$ consists simply of all the tangent spaces to the circles ${\mathcal O}_+(y)\cong \bbS^1$. 

Below we give an explanation of how the \emph{horizontal space} $\mathcal H$ in ${\rm T}\bbT(M)$ is defined, and what the \emph{horizontal lift} is. Having this explained, we will define $\mathcal D$ as in the statement of the proposition.

We start with the horizontal lift of vectors $Y$ from $M$ to $\bbT(M)$. It sends every tangent vector $Y_y$ from $y\in M$ to a vector $Y_{(y,\phi)}$ at a chosen point $(y,\phi)$ in the fiber $\pi^{-1}(y)$ as follows:

Take a curve $y(t)$ in $M$ starting at $y$, $y(0)=y$, and \emph{tangent} to $Y_y$. Then we identify the chosen point $(y,\phi)$ to which we want to lift our $Y_y$ with a totally null plane $N_{+}(y,\phi)$ in the tangent space $T_yM$. Using Levi-Civita connection associated with $g$ in $M$ we now \emph{parallel transport} the totally null plane $N_{+}(y,\phi)$ along the curve $y(t)$ from point $y$ to $y(t_f)$, with some $t_f>0$. In this way we obtain a \emph{curve} of 2-planes $\tilde{y}(t)=N_+(y,\phi,t)$ along $y(t)$ for all $0\leq t\leq t_f$. Since the Levi-Civita connection preserves nullity of vectors, the curve $\tilde{y}(t)$ of planes, is actually a curve of \emph{totally null planes}. And for sufficiently small $t_f$ these totally null planes are selfdual for the reason of continuity, since $N_+(y,\phi)$ was selfdual, and $y(t)$ is continuous.\\ 
This shows that given a differentiable curve $y(t)$ in $M$, starting at $y$ and tangent to $Y_y$, we have a corresponding curve $\tilde{y}(t)$ in $\bbT(M)$ starting at $(y,\phi)$. The tangent vector to this curve $\frac{\der \tilde{y}}{\der t}_{|t=0}$ is by definition the \emph{horizontal lift} $\tilde{Y}_{(y,\phi)}$ of $Y_y$ from $y$ to $(y,\phi)\in \bbT(M)$,
$$\tilde{Y}_{(y,\phi)}=\frac{\der \tilde{y}}{\der t}_{|t=0}.$$
It is a matter of checking that the construction of this lift does not depend on the choice of the curve $y(t)$: any other curve $y_1(t)$ passing through $y$ at $t=0$, and tangent to $Y_y$ produces the same lift. It also follows that the image $H_{(y,\phi)}$ of the lift map $(y,Y_y,\phi)\stackrel{\sim}{\mapsto}\tilde{Y}_{(y,\phi)}$, with $(y,\phi)$ fixed, is at each point $(y,\phi)\in\bbT$ a 4-dimensional vector space, which we denote by $H_{(y,\phi)}$. This, by definition is the \emph{horizontal vector space} at $(y,\phi)$, and we define $\mathcal H$ as 
$${\mathcal H}=\bigcup_{(y,\phi)\in\bbT(M)} H_{(y,\phi)}.$$
\end{proof}

\begin{definition}
The canonical horizontal rank two distribution $\mathcal D$ on $\bbT(M)$ defined in the Proposition \ref{ho} is called the \emph{twistor distribution}. 
\end{definition}

To give an explicit formula for the horizontal lift in terms of the coordinates $(y,\phi)$ on $\bbT(M)$ we introduce the Levi-Civita connection 1-forms $\Gamma^i_{~j}$, associated with the orthonormal coframe (\ref{met}). These are uniquely defined by 
$$\der \sigma^i+\Gamma^i_{~j}\dz\sigma^j=0,\quad {\rm and}\quad \Gamma_{ij}+\Gamma_{ji}=0,$$
where $\Gamma_{ij}=g_{ik}\Gamma^k_{~j}$, and $g_{ij}$ and $\sigma^i$ are given by (\ref{met}). Once the connection 1-forms $\Gamma^i_{~j}$ are determined by the coframe and the metric (\ref{met}), they define connection coefficients $\Gamma^i_{~jk}$ via 
$$\Gamma^i_{~j}=\Gamma^i_{~jk}\sigma^k.$$
Then an elementary (but lengthy) calculation, using the explanation about the horizontal lift given in the proof of Proposition \ref{ho}, leads to the following Lemma:

\begin{lemma}
In coordinates $(y,\phi)$ on $\bbT(M)$ adapted to the orthonormal coframe (\ref{met}), the formulas for the horizontal lifts of the frame vectors $(e_1,e_2,e_3,e_4)$ are:
$$
\tilde{e}_i=e_i+\Big(\Gamma^3_{~4i}-\Gamma^1_{~2i}+(\Gamma^1_{~4i}-\Gamma^2_{~3i})\cos\phi+(\Gamma^1_{~3i}+\Gamma^2_{~4i})\sin\phi\Big)\partial_\phi, \quad \forall i=1,2,3,4.
$$
In particular, the twistor distribution $\mathcal D$ is spanned by two vector fields $\tilde{X}_1$ and $\tilde{X}_2$ on $\bbT(M)$ given by:
\be\small{\begin{aligned}
&\tilde{X}_1=e_1+\cos\phi e_3+\sin\phi e_4+z_1\partial_\phi\\
&\tilde{X}_2=e_2-\sin\phi e_3+\cos\phi e_4+z_2\partial_\phi,\label{tdist}
\end{aligned}}
\ee
with the following `horizontal corrections' $z_1$ and $z_2$:
\be
\small{\begin{aligned}
z_1&=\Gamma^3_{~41}-\Gamma^1_{~21}+\cos\phi(\Gamma^3_{~43}-\Gamma^2_{~31}+\Gamma^1_{~41}-\Gamma^1_{~23})+\sin\phi(\Gamma^3_{~44}+\Gamma^2_{~41}+\Gamma^1_{~31}-\Gamma^1_{~24})+\\&\cos^2\phi(\Gamma^1_{~43}-\Gamma^2_{~33})+\cos\phi\sin\phi(\Gamma^2_{~43}-\Gamma^2_{~34}+\Gamma^1_{~44}+\Gamma^1_{~33})+\sin^2\phi(\Gamma^1_{~34}+\Gamma^2_{~44})\\&
\\z_2&=\Gamma^3_{~42}-\Gamma^1_{~22}+\cos\phi(\Gamma^3_{~44}-\Gamma^2_{~32}+\Gamma^1_{~42}-\Gamma^1_{~24})+\sin\phi(-\Gamma^3_{~43}+\Gamma^2_{~42}+\Gamma^1_{~32}+\Gamma^1_{~23})+\\&\cos^2\phi(\Gamma^1_{~44}-\Gamma^2_{~34})+\cos\phi\sin\phi(\Gamma^2_{~44}+\Gamma^2_{~33}-\Gamma^1_{~43}+\Gamma^1_{~34})-\sin^2\phi(\Gamma^1_{~33}+\Gamma^2_{~43}).
\end{aligned}}\label{zs}
\ee
\end{lemma}

The resemblance of the formulas (\ref{tdist}) for the twistor distribution $\mathcal D$ to the formulas (\ref{dist}) for the velocity distribution ${\mathcal D}_v$ of two rolling surfaces, together with the Proposition \ref{mike}, suggests to specialize our considerations to $M=\Sigma_1\times\Sigma_2$, with $g=g_1\oplus(-g_2)$, where $(\Sigma_1,g_1)$ and $(\Sigma_2,g_2)$ are the two rolling surfaces.

We then have the following theorem.

\begin{theorem}\label{main}
There is a natural identification 
$$C(\Sigma_1,\Sigma_2)\cong\bbT(\Sigma_1\times\Sigma_2)$$
between the configuration space $C(\Sigma_1,\Sigma_2)$ of two surfaces rolling on each other `without slipping or twisting', and the circle twistor bundle $\bbT(\Sigma_1\times\Sigma_2)$  over the split signature metric 4-manifold $(\Sigma_1\times\Sigma_2, g_1\oplus(-g_2))$, where $g_i$ is the metric on $\Sigma_i$.

Moreover, in this identification, the velocity space ${\mathcal D}_v$ of two surfaces rolling on each other 'without slipping or twisting' coincides with the twistor distribution $\mathcal D$ on $\bbT(M)$,
$${\mathcal D}_v=\mathcal D.$$     
\end{theorem} 
\begin{proof}
The identification is obtained by means of Proposition \ref{mike}:

First, given two surfaces $(\Sigma_1,g_1)$ and $(\Sigma_2,g_2)$ we form
a split signature 4-manifold $M=\Sigma_1\times\Sigma_2$ with the metric $g=g_1\oplus(-g_2)$, and its circle twistor bundle $\bbT(\Sigma_1\times\Sigma_2)$. 
Then, given a point $(x,\hat{x},\phi)$ in $C(\Sigma_1,\Sigma_2)$ we identify it with a totally null plane $N_+(y,\phi)={\rm Span}(e_1+\cos\phi e_3+\sin\phi e_4,e_2+\sin\phi e_3-\cos\phi e_4)$ at $y=(x,\hat{x})$ in $M$. Here $(e_1,e_2,e_3,e_4)$ is an orthonormal basis for $g$ corresponding to two orthonormal bases $(e_1,e_2)$ for $g_1$ and $(e_3,e_4)$ for $g_2$. Thus, given a point $(x,\hat{x},\phi)$ in $C(\Sigma_1,\Sigma_2)$ we have totally null plane $N_+(y,\phi)$ at $y=(x,\hat{x})$ in $M=\Sigma_1\times\Sigma_2$, i.e. a point $(y,\phi)$ in $\bbT(\Sigma_1,\Sigma_2)$. 

Conversely, having $\bbT(\Sigma_1\times\Sigma_2)$ we can canonically split every projection $y=\pi((y,\phi))$ onto $y=(x,\hat{x})$, such that $x\in\Sigma_1$ and $\hat{x}\in\Sigma_2$. Since we have an interpretation of $(y,\phi)$ as a totally null plane $N_+(y,\phi)$ at $y=(x,\hat{x})$ we can now associate to it $A_\phi$ as a unique linear orthogonal map $A_\phi: {\rm T}_x\Sigma_1\to {\rm T}_{\hat{x}}\Sigma_2$ whose graph in 
$({\rm T}_x\Sigma_1)\times({\rm T}_{\hat{x}}\Sigma_2)$ is the totally null plane $N_+(y,\phi)$. 

This shows both directions of the identification.

Having given the identification, we now specialize the formula (\ref{tdist}) to the case when $M=\Sigma_1\times\Sigma_2$. We return to the setting as in formula (\ref{fr}), where the orthonormal frames $(e_1,e_2)$ and $(e_3,e_4)$ are extended to the orthonormal frame $(e_1,e_2,e_3,e_4)$ in $M=\Sigma_1\times\Sigma_2$. Now, having the commutation relations (\ref{fr}) we calculate the connection coefficients $\Gamma^i_{~jk}$ of the Levi-Civita connection of $g=g_1\oplus(-g_2)$ in the frame $(e_1,e_2,e_3,e_4)$. These are:
$$\Gamma^1_{~21}=a_1,\quad\Gamma^1_{~22}=a_2,\quad\Gamma^3_{~43}=a_3,\quad\Gamma^3_{~44}=a_4.$$
Modulo the symmetry, $g_{ij}\Gamma^j_{~kl}=-g_{kj}\Gamma^j_{~il}$, all other connection coefficients are zero. These, when inserted in the expressions (\ref{zs}) for the horizontal corrections $z_1$ and $z_2$, give:
$$z_1=-a_1+a_3\cos\phi+a_4\sin\phi\quad {\rm and}\quad z_2=-a_2+a_4\cos\phi-a_3\sin\phi.$$
Insertion of these $z_1$ and $z_2$ into formulas (\ref{tdist}) defining the vectors $\tilde{X}_1$ and $\tilde{X}_2$,  transforms the vectors spanning the twistor distribution $\mathcal D$ into Agrachov-Sachkov's vectors (\ref{dist}) spanning the velocity space ${\mathcal D}_v$ of the two rolling surfaces restricted by the non-slipping and non-twisting conditions.

This finishes the proof.
\end{proof}

In view of this theorem we have the following definition:
\begin{definition}
Let $(\Sigma_1,g_1)$ and $(\Sigma_2,g_2)$ be two Riemann surfaces. The circle twistor bundle $\bbT(M)$ over a manifold $M=\Sigma_1\times\Sigma_2$ equipped with the split-signature metric $g=g_1\oplus(-g_2)$ is called a \emph{twistor space} for the surfaces $\Sigma_1$ and $\Sigma_2$ that roll on each other `without slipping or twisting'. 
\end{definition}
\section{Cartan's invariants of rank two distributions in dimension five}\label{invsec}
For the completeness we will now present the basic, well known, or implicit in Refs. \cite{brya,cartan,nurdif}, facts about rank two distributions in dimensions five, which will be needed in the next Section. This part of the paper is purely expository, and it is based on Ref. \cite{nurdif}. The reader is referred to this paper for details.

Let $\tilde{X}_1$ and $\tilde{X}_2$ be two linearly independent vector fields on a 5-dimensional manifold $M^5$. Their span 
$${\mathcal D}={\rm Span}(\tilde{X}_1,\tilde{X}_2)$$
is a \emph{rank two distribution} on $M^5$. If $[\tilde{X}_1,\tilde{X}_2]=a_1\tilde{X}_1+a_2\tilde{X}_2$ for some functions $a_1$, $a_2$ on $M^5$, the distribution is \emph{integrable}. Such distributions do not have \emph{local invariants}, in the sense that every such distribution can be locally brought to the form ${\mathcal D}={\rm Span}(\partial_x,\partial_q)$, by a local diffeomorphism of $M^5$. On the other extreme, a rank two distribution is called \emph{generic}, or $(2,3,5)$, as e.g. in \cite{dub,spar}, if we have:
\be
[\tilde{X}_1,\tilde{X}_2]=\tilde{X}_3,\quad [\tilde{X}_1,\tilde{X}_3]=\tilde{X}_4,\quad [\tilde{X}_2,\tilde{X}_3]=\tilde{X}_5,\label{235}\ee
and at each point of $M^5$ the five vectors $(\tilde{X}_1,\tilde{X}_2,\tilde{X}_3,\tilde{X}_4,\tilde{X}_5)$ are \emph{linearly independent}. 

Generic rank two distributions in dimension five have nontrivial local invariants - in general given two $(2,3,5)$ distributions ${\mathcal D}_1$ and ${\mathcal D}_2$ on $M^5$ a local diffeomorphism $\varphi:M^5\to M^5$ such that $\varphi_*{\mathcal D}_1={\mathcal D}_2$ does \emph{not} exist. If we have a $(2,3,5)$ distribution $\mathcal D$ on $M^5$ for which we have a (local) 
diffeomorphism $\varphi:M^5\to M^5$ such that $\varphi_*{\mathcal D}=\mathcal D$, we say that $\mathcal D$ has a \emph{symmetry} $\varphi$. The full set of local symmetries for $\mathcal D$ is locally a Lie group, the \emph{symmetry group} of $\mathcal D$, which locally can be described by its Lie algebra, realized as a Lie algebra of vector fields $Y$ on $M^5$ such that $[Y,{\mathcal D}]\subset{\mathcal D}$.      

It turns out that $(2,3,5)$ distribution $\mathcal D$ with \emph{maximal group of symmetries} is locally diffeomorphic to
$${\mathcal D}_{G_2}={\rm Span}(~\partial_x+p\partial_y+q\partial_p+\tfrac12 q^2\partial_z,~~~\partial_q~),$$
where $(x,y,p,q,z)$ are local coordinates on $M^5$. It is a result of E. Cartan and F. Engel, \cite{cart,engel}, that in this case the local symmetry group is isomorphic to the split real form of the exceptional Lie group $G_2$. Thus the maximal group of local symmetries for a $(2,3,5)$ distribution has dimension 14. 

E. Cartan in \cite{cartan} gave a necessary and sufficient condition for a $(2,3,5)$ distribution $\mathcal D$ to be locally diffeomorphic to ${\mathcal D}_{G_2}$. For this a certain quartic, the \emph{Cartan quartic},  
\be
C(\zeta)=A_1+4A_2\zeta+6A_3\zeta^2+4A_4\zeta^3+A_5\zeta^4,\label{cq}\ee
with certain functions $A_1,A_2,A_3,A_4,A_5$ on $M^5$, have to identically vanish. This means that an if and only if condition for a distribution $\mathcal D$ 
to be locally diffeomorphic to ${\mathcal D}_{G_2}$ in a neighbourhood of a point is the vanishing of \emph{all} $A_i$s: 
$$A_1\equiv A_2\equiv A_3\equiv A_4\equiv A_5\equiv 0,$$
in this neighbourhood. This gives us an important corollary of the identification theorem (\ref{main}).

\begin{corollary}\label{maincor}
The velocity space ${\mathcal D}_v$ of two surfaces rolling on each other 'without slipping or twisting' has local symmetry group $G_2$ around a point if and only if the Cartan quartic of the circle twistor bundle $\bbT(M)$ identically vanishes in a neighborhood of the point. 
\end{corollary}

In the procedure below, which is implicit in \cite{nurdif}, and more explicit in \cite{Gra}, we summarize how to effectively calculate $C(\zeta)$ given a $(2,3,5)$ distribution $\mathcal D$ on $M^5$. In particular, we show how to calculate the functions $A_i$.
\subsection{A procedure for calculating Cartan's quartic}
Let ${\mathcal D}={\rm Span}(\tilde{X}_1,\tilde{X}_2)$ be a $(2,3,5)$ distribution on a 5-dimensional manifold $M^5$. 
\begin{itemize}
\item Form the vectors $\tilde{X}_3,\tilde{X}_4,\tilde{X}_5$ by taking the appropriate commutators as in (\ref{235}). Since the distribution is $(2,3,5)$ the vector fields $(\tilde{X}_1,\tilde{X}_2,\tilde{X}_3,\tilde{X}_4,\tilde{X}_5)$ constitute a local frame on $M^5$.
\item Consider the coframe $(\om_1,\om_2,\om_3,\om_4,\om_5)$ of 1-forms dual to 
$(\tilde{X}_1,\tilde{X}_2,\tilde{X}_3,\tilde{X}_4,\tilde{X}_5)$. This means that the forms $\om_i$ are related to vector fields $\tilde{X}_j$ via:
$$\tilde{X}_i\hook \omega_j=\delta_{ij}.$$
\item 
Introduce the `invariant forms' $(\theta^1,\theta^2,\theta^3,\theta^4,\theta^5)$ defined by 
\be
\bma \theta^1\\\theta^2\\\theta^3\\\theta^4\\\theta^5\ema =\bma b_{11}&b_{12}&b_{13}&0&0\\b_{21}&b_{22}&b_{23}&0&0\\b_{31}&b_{32}&b_{33}&0&0\\
b_{41}&b_{42}&b_{43}&b_{44}&b_{45}\\
b_{51}&b_{52}&b_{53}&b_{54}&b_{55}\ema\bma\om_3\\\om_4\\\om_5\\\om_1\\\om_2\ema,\label{inf}\ee
with some unknown functions $b_{\mu\nu}$ on $M^5$ satisfying the nonvanishing determinant condition:
$$(b_{13}b_{23}b_{31}-b_{12}b_{23}b_{31}-b_{13}b_{21}b_{32}+b_{11}b_{23}b_{32}+b_{12}b_{21}b_{33}-b_{11}b_{22}b_{33})(b_{45}b_{54}-b_{44}b_{55})\neq 0.$$
\item Force these forms to satisfy the exterior differential system
\be\begin{aligned}
&\der\theta^1=\theta^1\dz(2\Om_1+\Om_4)+\theta^2\dz\Om_2+\theta^3\dz\theta^4\\
&\der\theta^2=\theta^1\dz\Om_3+\theta^2\dz(\Om_1+2\Om_4)+\theta^3\dz\theta^5\\
&\der\theta^3=\theta^1\dz\Om_5+\theta^2\dz\Om_6+\theta^3\dz(\Om_1+\Om_4)+\theta^4\dz\theta^5\\
&\der\theta^4=\theta^1\dz\Om_7+\tfrac43 \theta^3\dz\Om_6+\theta^4\dz\Om_1+\theta^5\dz\Om_2\\
&\der\theta^5=\theta^2\dz\Om_7-\tfrac43 \theta^3\dz\Om_5+\theta^4\dz\Om_3+\theta^5\dz\Om_4,
\end{aligned}\label{carsys}
\ee
with some 1-forms $(\Om_1,\Om_2,\dots,\Om_7)$. This, in particular, will impose conditions on the unknowns $b_{\mu\nu}$ that should be solved.
\item It follows (and this is explained in full detail in \cite{cartan}, see also \cite{nurdif}) that given $\tilde{X}_1$ and $\tilde{X}_2$ spanning a $(2,3,5)$ distribution, all the above mentioned conditions on $b_{\mu\nu}$ are algebraic, and can be always \emph{explicitly} solved. Consequently the forms $\theta^1,\theta^2,\dots,\theta^5$ and $\Om_1,\Om_2,\dots,\Om_7$ can be \emph{explicitly} found. One has to note, however, that the equations (\ref{carsys}) do \emph{not} determine \emph{all} the unknown 
coefficients $b_{\mu\nu}$, and that, as a consequence, the forms $\theta^1,\theta^2,\dots,\theta^5$, $\Om_1,\Om_2,\dots,\Om_7$ are not uniquely specified. In particular, the forms $\theta^1,\theta^2,\dots,\theta^5$ still depend on the undetermined $b_{\mu\nu}$s, and $\Om_1,\Om_2,\dots,\Om_7$, apart from depending on these 
$b_{\mu\nu}$s, are given up to additional freedom. It follows that this freedom, i.e. not totally determined $b_{\mu\nu}$s and the additional freedom in the choice of $\Omega_A$s, is not relevant, for finding the \emph{zeros} of the Cartan quartic: an important observation of Cartan is that under the transformations induced by this freedom Cartan's tensor \emph{merely scales} by a nonvanishing function.   
\item Thus, given $\tilde{X}_1$ and $\tilde{X}_2$, find a representative of your choice of the forms
$(\theta^1,\theta^2,\theta^3,\theta^4,\theta^5)$ as in (\ref{inf}) satisfying (\ref{carsys}).  Because of the determinant conditions satisfied by the $b_{\mu\nu}$s the 1-forms $(\theta^1,\theta^2,\theta^3,\theta^4,\theta^5)$  constitute a \emph{coframe} on $M^5$.
\item Construct a $(3,2)$ signature bilinear form $\tilde{g}$ on $M^5$ given by
\be
\boxed{\tilde{g}=\theta^1\otimes\theta^5+\theta^5\otimes\theta^1-\theta^2\otimes\theta^4-\theta^4\otimes\theta^2+\tfrac43\theta^3\otimes\theta^3.}\label{conf}
\ee    
It was shown in \cite{nurdif} that $\tilde{g}$ transforms conformally under the transformations induced by the freedom in the choice of $\theta^i$s, and therefore it defines a conformal class $[\tilde{g}]$ of $(3,2)$ signature metrics on $M^5$. This class is entirely determined by the the distribution ${\mathcal D}={\rm Span}(\tilde{X}_1,\tilde{X}_2)$. 
\item Consider vector fields $(Y_1,Y_2,Y_3,Y_4,Y_5)$ on $M^5$ which are dual, $$\boxed{Y_i\hook\theta^j=\delta^j_{~i},}$$ to the coframe 1-forms $(\theta^1,\theta^2,\theta^3,\theta^4,\theta^5)$. In terms of these vectors the distribution $\mathcal D$ is spanned by the vectors $Y_4$ and $Y_5$,
$${\mathcal D}={\rm Span}(Y_4,Y_5).$$
As it is easily seen $\mathcal D$ is \emph{totally null} in the conformal class $[\tilde{g}]$. Also the distribution ${\mathcal E}={\rm Span}(Y_1,Y_2)$ is totally null in $[\tilde{g}]$.
\item It turns out that for every null vector field $Z_1$ in $\mathcal D$ there is precisely one null line $\bbR\cdot Z_2$ in $\mathcal E$ orthogonal to it. Indeed,
if $Z_1=\alpha Y_4+\beta Y_5$ then the unique orthogonal line in $\mathcal E$ is spanned by $Z_2=\alpha Y_1+\beta Y_2$. Ignoring the situation when $\alpha=0$, we introduce a coordinate $\zeta=\frac{\beta}{\alpha}$ parametrizing both lines. Thus, to a null line in $\mathcal D$ we have a unique null line in $\mathcal E$.
For each value of $\zeta$ they are respectively spanned by 
$$\boxed{Z_1(\zeta)=Y_4+\zeta Y_5\quad{\rm and}\quad Z_2(\zeta)=Y_1+\zeta Y_2.}$$
\item Choose a simple representative $\tilde{g}_0$ of the conformal class, consider its Weyl tensor $\tilde{C}^i_{~jkl}$ and lower the index $i$ 
by $\tilde{g}_0$ to have $\tilde{C}_{ijkl}$. This enables to think about $\tilde{C}(\cdot,\cdot,\cdot,\cdot)$ as a multilinear map $$\boxed{\tilde{C}(\cdot,\cdot,\cdot,\cdot)~:~{\rm T}M^5\times{\rm T}M^5\times{\rm T}M^5\times{\rm T}M^5\to {\mathcal F}(M^5),}$$
where ${\mathcal F}(M^5)$ denotes the set of smooth functions on $M^5$.
\item Implicit in \cite{nurdif} is the formula
$$\boxed{\begin{aligned}
C(\zeta)~:=~&A_1+4A_2\zeta+6A_3\zeta^2+4A_4\zeta^3+A_5\zeta^4~=\\&~h~\tilde{C}(~Z_1(\zeta),~Z_2(\zeta),~Z_1(\zeta),~Z_2(\zeta)~),
\end{aligned}}$$
where $h$ is a nonvanishing function on $M^5$.
\item Thus the Cartan quartic (\ref{cq}) is, modulo a nonvanishing factor, the quantity: $\tilde{C}(Z_1(\zeta),Z_2(\zeta),$ $Z_1(\zeta),Z_2(\zeta))$, obtained from a pair $(Z_1(\zeta),Z_2(\zeta))$ of null directions $Z_1(\zeta)$ in $\mathcal D$, and  the corresponding orthogonal null directions $Z_2(\zeta)$ in $\mathcal E$, and from the Weyl tensor $\tilde{C}$ of the conformal class $[\tilde{g}]$. In particular, the functions $A_i$ whose vanishing is necessary and sufficient for $\mathcal D$ to have local symmetry $G_2$, modulo nonvanishing factors, are given by:
\be\boxed{\begin{aligned}
A_1=&\tilde{C}(Y_4,Y_1,Y_1,Y_4),\quad  A_2=\tilde{C}(Y_4,Y_1,Y_2,Y_4),\quad A_3=\tilde{C}(Y_4,Y_1,Y_2,Y_5),\\&A_4=\tilde{C}(Y_4,Y_2,Y_2,Y_5),\quad   A_5=\tilde{C}(Y_5,Y_2,Y_2,Y_5).\end{aligned}}\label{cara}\ee
\item We note here a theorem of Cartan that the vanishing of $A_i$'s is actually a necessary and sufficient condition for all Weyl tensor $\tilde{C}^i_{~jkl}$ to vanish. Thus in the Corollary (\ref{maincor}), vanishing of Cartan quartic can be replaced by vanishing of Weyl tensors, i.e. that $(\bbT (M),[\tilde{g}])$ is conformally flat.
\end{itemize}

\section{Examples of surfaces whose twistor distribution has $G_2$ symmetry}
\subsection{The problem}
The restricted velocity space ${\mathcal D}_v$ for two balls (bounded by the two spheres $\bbS^2_{r_1}$ and $\bbS^2_{r_2}$ of the respective radii $r_1$ and $r_2$)  rolling on each other `without slipping or twisting' has been investigated for a while during recent years, see \cite{agr1,brya,monty,zele}. It is therefore well known that the distribution ${\mathcal D}_v$ defined by such a system on the configuration space $C(\bbS^2_{r_1},\bbS^2_{r_2})$ is \emph{integrable} if and only if the radii $r_1$ and $r_2$ of the balls, are equal. In case when the radii are not equal the distribution ${\mathcal D}_v$ is $(2,3,5)$ and has always a \emph{global} symmetry $\sog(3)\times\sog(3)$. But a surprising result of R. Bryant/G. Bor/R. Montgomery/I. Zelenko, \cite{monty,zele}, says that if, in addition to $r_1\neq r_2$, the ratio of the radii is $r_1:r_2=3$ or $r_1:r_2=\tfrac13$, then the $(2,3,5)$ distribution ${\mathcal D}_v$ has the \emph{maximal} local symmetry in the non-integrable case, in which case the local symmetry group is $G_2$. This remarkable observation gives a `physical' realization of this exceptional Lie group; a realization unnoticed by mathematicians and physicists for more than 100 years, from the year 1894, when E. Cartan and F. Engel, have shown that this group is a symmetry group of a certain rank two distribution in dimension five \cite{cart,engel}. 

The peculiar $3:1$ or $1:3$ ratio of the radii of the two balls for which ${\mathcal D}_v$ has local symmetry $G_2$ provoked the question posed by G. Bor and R. Montgomery, for a `geometric' explanation of this fact. In our opinion this question would be very interesting if the two balls with these ratios were the \emph{only} two surfaces which rolling on each other `without slipping or twisting' had $G_2$ as the local symmetry group. 
The aim of the rest of the paper is to show that this is \emph{not} the case: we are able to find surfaces that roll `without slipping or twisting' on a \emph{plane} having ${\mathcal D}_v$ with local symmetry $G_2$. Thus, in view of the result we are going to present in this section, we propose to change the question of R. Bor and R. Montgomery into the following problem:

{\it Find all the pairs of surfaces which when rolling on each other `without slipping or twisting' having the velocity space as a (2,3,5) distribution ${\mathcal D}_v$ with $G_2$ as the local group of symmetries.}

\subsection{General setting}
Before passing to our examples we set the framework for the problem in the full generality, when we have two general 
surfaces $(\Sigma_1,g_1)$ and $(\Sigma_2,g_2)$.

According to Theorem \ref{main} we identify the configuration space $C(\Sigma_1,\Sigma_2)$ with the twistor space $\bbT(\Sigma_1\times\Sigma_2)$ of the manifold $M=\Sigma_1\times\Sigma_2$ with metric $g=g_1\oplus(-g_2)$. We choose the corresponding orthonormal frames $(e_1,e_2)$ on $\Sigma_1$ and $(e_3,e_4)$ on $\Sigma_2$, extend them to $M$ as it was explained below the formula (\ref{fr}), and write down the generators of the twistor distribution $\mathcal D$ as in (\ref{tdist})-(\ref{zs}). We now introduce the following notation:
$$e_i\hook\der f=:f_i,$$
which associates a lower index $i$ to a frame derivative of a function $f$ in the direction of the frame vector $e_i$, $i=1,2,3,4$, on $\bbT(M)$. With this notation, the Gaussian curvatures $\kappa$ of $g_1$ and $\lambda$ of $g_2$ are:
$$\kappa=a_{21}-a_{12}-a_1^2-a_2^2,\quad\quad \lambda=a_{43}-a_{34}-a_3^2-a_4^2.$$
Now we calculate the commutator $[\tilde{X}_1,\tilde{X}_2]$, which turns out to be 
$$[\tilde{X}_1,\tilde{X}_2]=\tilde{X}_3,$$
where
\be
\tilde{X}_3=a_1\tilde{X}_1+a_2\tilde{X}_2+(\lambda-\kappa)\partial_\phi.\label{x3}\ee
Thus, the twistor distribution is integrable if and only if the two surfaces have equal curvatures, as it was claimed. 
From now on, we will only deal with the surfaces with unequal curvatures:
$$\kappa\neq \lambda.$$
The next step is to calculate the commutators $[\tilde{X}_1,\tilde{X}_3]$ and $[\tilde{X}_2,\tilde{X}_3]$. We denote the results by 
$\tilde{X}_4$ and $\tilde{X}_5$, respectively:
$$[\tilde{X}_1,\tilde{X}_3]=\tilde{X}_4,\quad\quad[\tilde{X}_2,\tilde{X}_3]=\tilde{X}_5.$$
The (ugly) formulas for $\tilde{X}_4$ and $\tilde{X}_5$ are:
\be\begin{aligned}
\tilde{X}_4=&\Big(a_{11}+a_1\big(\frac{\kappa_1}{\lambda-\kappa}-(a_3+\frac{\lambda_4}{\lambda-\kappa})\sin\phi+(a_4-\frac{\lambda_3}{\lambda-\kappa})\cos\phi\big)\Big)\tilde{X}_1+\\
&\Big(a_{21}+a_2\big(\frac{\kappa_1}{\lambda-\kappa}-(a_3+\frac{\lambda_4}{\lambda-\kappa})\sin\phi+(a_4-\frac{\lambda_3}{\lambda-\kappa})\cos\phi\big)\Big)\tilde{X}_2-\\
&\Big(\frac{\kappa_1}{\lambda-\kappa}-(a_3+\frac{\lambda_4}{\lambda-\kappa})\sin\phi+(a_4-\frac{\lambda_3}{\lambda-\kappa})\cos\phi\Big)\tilde{X}_3+\\
&(\lambda-\kappa)\Big(\sin\phi e_3-\cos\phi e_4\Big),
\end{aligned}\label{x4}\ee
and
\be\begin{aligned}
\tilde{X}_5=&\Big(a_{12}+a_1\big(\frac{\kappa_2}{\lambda-\kappa}-(a_4-\frac{\lambda_3}{\lambda-\kappa})\sin\phi-(a_3+\frac{\lambda_4}{\lambda-\kappa})\cos\phi\big)\Big)\tilde{X}_1+\\
&\Big(a_{22}+a_2\big(\frac{\kappa_3}{\lambda-\kappa}-(a_4-\frac{\lambda_3}{\lambda-\kappa})\sin\phi-(a_3+\frac{\lambda_4}{\lambda-\kappa})\cos\phi\big)\Big)\tilde{X}_2-\\
&\Big(\frac{\kappa_2}{\lambda-\kappa}-(a_4-\frac{\lambda_3}{\lambda-\kappa})\sin\phi-(a_3+\frac{\lambda_4}{\lambda-\kappa})\cos\phi\Big)\tilde{X}_3+\\
&(\lambda-\kappa)\Big(\cos\phi e_3+\sin\phi e_4\Big).
\end{aligned}\label{x5}\ee
These equations show, in particular, that if $\kappa\neq\lambda$, 
the five vector fields $(\tilde{X}_1,\tilde{X}_2,\tilde{X}_3,$ $\tilde{X}_4,\tilde{X}_5)$, form a \emph{frame} on $\bbT(M)$, and that in such case the twistor distribution is always a $(2,3,5)$. Now, to analyze the invariants of $\mathcal D$ it is convenient to pass from the `surfaces adapted frame' $(e_1,e_2,e_3,e_4,\partial_\phi)$ on $\bbT(M)$ to the adapted frame $(\tilde{X}_1,\tilde{X}_2,\tilde{X}_3,\tilde{X}_4,\tilde{X}_5)$, and use the procedure outlined in Section \ref{invsec}. Passing to the duals $(\om_1,\om_2,\om_3,\om_4,\om_5)$ of $(\tilde{X}_1,\tilde{X}_2,\tilde{X}_3,$ $\tilde{X}_4,\tilde{X}_5)$ and considering the invariant forms $(\theta^1,\theta^2,\theta^3,\theta^4,\theta^5)$ as in (\ref{inf}) we find that the unknowns $b_{\mu\nu}$ must, in particular, satisfy the following equations:
$$\begin{aligned}
b_{11}=&b_{21}=0,\quad b_{44}=\frac{b_{12}}{b_{31}},\quad b_{45}=\frac{b_{13}}{b_{31}},\quad b_{54}=\frac{b_{22}}{b_{31}},\\
&b_{23}=\frac{b_{13}b_{22}-b_{31}^3}{b_{12}},\quad b_{55}=\frac{b_{13}b_{22}-b_{31}^3}{b_{12}b_{31}},\end{aligned}$$ 
$$\begin{aligned}
&b_{41}=
\Big(\big(4b_{13}b_{32}-4b_{12}b_{33}+3b_{31}(a_1b_{12}+a_2b_{13})\big)(\kappa-\lambda)+3b_{31}(b_{12}\kappa_2-b_{13}\kappa_1)+\\
&3b_{31}\big((b_{13}\lambda_3-b_{12}\lambda_4)\cos\phi+(b_{12}\lambda_3+b_{13}\lambda_4)\sin\phi\big)\Big)\Big(3b_{31}^2(\kappa-\lambda)\Big)^{-1}.
\end{aligned}$$ 
$$\begin{aligned}
&b_{51}=
\Big(\big(3(a_1b_{12}+a_2b_{13})b_{22}b_{31}-b_{31}^3(3a_2b_{31}+4b_{32})+4b_{22}(b_{13}b_{32}-b_{12}b_{33})\big)(\kappa-\lambda)+
\\&3b_{31}\big(b_{22}(b_{12}\kappa_2-b_{13}\kappa_1)+b_{31}^3\kappa_1+(b_{22}(b_{13}\lambda_3-b_{12}\lambda_4)-b_{31}^3\lambda_3)\cos\phi+
\\&(b_{22}(b_{12}\lambda_3+b_{13}\lambda_4)-b_{31}^3\lambda_4)\sin\phi\big)\Big)(1+
\cos\phi)\Big(6b_{12}b_{31}^2(\kappa-\lambda)\Big)^{-1}.
\end{aligned}$$ 
We have also obtained formulas for $b_{42}$, $b_{52}$ and $b_{53}$. Their length prevents us from displaying them here. 
The important thing is that the equations (\ref{carsys}), and our formulas for $b_{\mu\nu}$ implied by them, enabled us to find an explicit 
representative for the conformal class $[\tilde{g}]$ discussed in Section \ref{invsec}. We have also calculated the coefficients $A_1,A_2,A_3,A_4,A_5$ of the Cartan quartic in this general case. The formulas for them are very long and not very illuminating. So we will not display them here. Instead, we concentrate on special cases.

\subsection{Surface with one Killing vector and a surface of constant curvature}\label{killi}
To simplify the matters we consider a surface $(\Sigma_1,g_1)$ with a \emph{Killing vector} rolling on a \emph{surface}  
$(\Sigma_2,g_2)$ \emph{of constant Gaussian curvature}. 

We aim to find all pairs $(g_1,g_2)$ for which the corresponding twistor distributions $\mathcal D$ has local symmetry $G_2$.

We use the setting from the previous section. Since $g_2$ is a metric of constant curvature $\lambda$, and $g_1$ has Killing symmetry, our assumptions enable us to choose $(e_1,e_2,e_3,e_4)$ such that:
$$a_1=0,\quad a_3=0,\quad a_{21}=\kappa+a_2^2,\quad a_{22}=0,\quad a_{43}=\lambda+a_4^2,\quad a_{44}=0,\quad\der\lambda=0.$$
In the above, we have assumed that the Killing vector field of $g_1$ is the vector field $e_2$ multiplied by a suitable positive smooth function on $\Sigma_1$, so in particular $\kappa_2=0$. The metrics $g_1$ and $g_2$ read:
$$g_1=(\sigma^1)^2+(\sigma^2)^2,\quad\quad g_2=(\sigma^3)^2+(\sigma^4)^2,$$
where $(\sigma^1,\sigma^2)$, and $(\sigma^3,\sigma^4)$ are the respective duals to $(e_1,e_2)$ and $(e_3,e_4)$. The split signature metric $g$, in this setup reads:
$$g=(\sigma^1)^2+(\sigma^2)^2-(\sigma^3)^2-(\sigma^4)^2,$$
and we have
$$\der\sigma^1=0,\quad\der\sigma^2=-a_2\sigma^1\dz\sigma^2,\quad\der\sigma^3=0,\quad\der\sigma^4=-a_4\sigma^3\dz\sigma^4.$$
These assumptions enormously simplify the expression for the conformal metric $\tilde{g}$.

We have the following proposition.
\begin{proposition}\label{cocu}
The conformal $(3,2)$-signature class $[\tilde{g}]$ associated with the twistor distribution $\mathcal D$ of the twistor space $\bbT(\Sigma_1\times\Sigma_2)$ for two surfaces, the first with a Killing vector, and the second a space of constant Gaussian curvature $\lambda$,  
is represented by the metric \be\tilde{g}=\theta^1\otimes\theta^5+\theta^5\otimes\theta^1-\theta^2\otimes\theta^4-\theta^4\otimes\theta^2+\tfrac43\theta^3\otimes\theta^3\label{conf1}\ee with the basis 1-forms $(\theta^1,\theta^2,\theta^3,\theta^4,\theta^5)$ given by:
$$\theta^1=\om_4-\om_5,\quad\quad\theta_2=\om_5,\quad\quad\theta^3=-\om_3,$$
$$\begin{aligned}
\theta^4=&-\om_1+\om_2+(a_2+\frac{\kappa_1}{\lambda-\kappa})\om_3+(a_2^2+\tfrac85\kappa-\tfrac75\lambda+\tfrac{1}{10}\frac{\kappa_{11}-a_2\kappa_1}{\kappa-\lambda}-\tfrac12\frac{\kappa_1^2}{(\kappa-\lambda)^2})\om_4,\\
\theta^5=&-\om_2-(a_2+\frac{\kappa_1}{\lambda-\kappa})\om_3-(a_2^2+\tfrac{13}{10}\kappa-\tfrac{7}{10}\lambda+\tfrac{1}{10}\frac{\kappa_{11}}{\kappa-\lambda}-\tfrac12\frac{\kappa_1^2}{(\kappa-\lambda)^2})\om_4+\\&(\tfrac{3}{10}\kappa-\tfrac{7}{10}\lambda+\tfrac{1}{10}\frac{a_2\kappa_1}{\lambda-\kappa})\om_5,
\end{aligned}$$
with the basis forms $(\om_1,\om_2,\om_3,\om_4,\om_5)$, which are the duals of the vector fields (\ref{tdist})-(\ref{zs}), (\ref{x3}), (\ref{x4}), (\ref{x5}), given by:
$$
\begin{aligned}
\om_1=&\sigma^1,\\
\om_2=&(2a_2^2\kappa+2\kappa^2-a_2\kappa_1-2a_2^2\lambda-3\kappa\lambda+\lambda^2)\frac{\sigma^2}{(\kappa-\lambda)^2}+\\&(a_2^2\kappa+\kappa^2-a_2\kappa_1-a_2^2\lambda-\kappa\lambda)\sin\phi \frac{\sigma^3}{(\kappa-\lambda)^2}-\\&\big(a_2a_4(\kappa-\lambda)+(a_2^2\kappa+\kappa^2-a_2\kappa_1-a_2^2\lambda-\kappa\lambda)\cos\phi\big)\frac{\sigma^4}{(\kappa-\lambda)^2}+a_2\frac{\der\phi}{\kappa-\lambda},\\
\om_3=&\big(-a_2(\kappa-\lambda)+\kappa_1\big)\frac{\sigma^2}{(\kappa-\lambda)^2}+\kappa_1\sin\phi \frac{\sigma^3}{(\kappa-\lambda)^2}+\\&\big(a_4(\kappa-\lambda)-\kappa_1\cos\phi\big)\frac{\sigma^4}{(\kappa-\lambda)^2}-\frac{\der\phi}{\kappa-\lambda},\\
\om_4=&-\frac{\sigma^2}{\kappa-\lambda}-\sin\phi\frac{\sigma^3}{\kappa-\lambda}+\cos\phi\frac{\sigma^4}{\kappa-\lambda},\\
\om_5=&\frac{\sigma^1}{\kappa-\lambda}-\cos\phi\frac{\sigma^3}{\kappa-\lambda}-\sin\phi\frac{\sigma^4}{\kappa-\lambda}.\\
\end{aligned}
$$
\end{proposition}
To answer the question of when $\mathcal D$ of Proposition \ref{cocu} has $G_2$ as local group of symmetries, we need only particular components of the Weyl tensor of the metric (\ref{conf1}). However, we declare that we were able to calculate the entire Weyl tensor in a manageable form. In particular we have found all the components $(A_1,A_2,A_3,A_4,A_5)$ of the Cartan's quartic in this case. They read:
\be\begin{aligned}
A_1=&A_2=\\
&10(\kappa-\lambda)^3\kappa_{1111}-70(\kappa-\lambda)^2\kappa_{111}\kappa_1-49(\kappa-\lambda)^2\kappa_{11}^2+280(\kappa-\lambda)\kappa_1^2\kappa_{11}+\\&8(\kappa-\lambda)^3(2\kappa+7\lambda)\kappa_{11}-20(\kappa-\lambda)^2(\kappa+6\lambda)\kappa_1^2-175\kappa_1^4+\\&(\kappa-\lambda)^4(\kappa-9\lambda)(9\kappa-\lambda),\\
A_3=&A_1-
10(\kappa-\la)^3a_2\kappa_{111}+\tfrac{154}{3}(\kappa-\lambda)^2a_2\kappa_{11}\kappa_1-20(\kappa-\lambda)^3a_2^2\kappa_{11}-\\&\tfrac43(\kappa-\lambda)^3(3\kappa-7\lambda)\kappa_{11}-\tfrac{140}{3}(\kappa-\lambda)a_2\kappa_1^3+\tfrac53(\kappa-\lambda)^2(21a_2^2+4\kappa-11\lambda)\kappa_1^2-\\&\tfrac43(\kappa-\lambda)^3(15a_2^2+12\kappa+7\lambda)a_2\kappa_1+\tfrac13(\kappa-\lambda)^4(\kappa-9\lambda)(9\kappa-\lambda),\\
A_4=&-2A_1+3A_3,\\
A_5=&-5A_1+6A_3+30(\kappa-\lambda)^3a_2^2\kappa_{11}-49(\kappa-\lambda)^2a_2^2\kappa_1^2+\\&2(\kappa-\lambda)^3(15a_2^2-3\kappa-28\lambda)a_2\kappa_1+(\kappa-\lambda)^4(\kappa-9\lambda)(9\kappa-\lambda).\label{eqg2}
\end{aligned}\ee
We have the following theorem.

\begin{theorem}\label{consts}
Let $(\Sigma_1,g_1)$ be a Riemann surface with Gaussian curvature $\kappa$, which has a \emph{Killing vector}, and let $(\Sigma_2,g_2)$ be a Riemann \emph{surface of constant Gaussian curvature} $\lambda$. Consider configuration space of the two surfaces rolling on each other `without slipping or twisting'. Then in order for distribution ${\mathcal D}_v$  to have local symmetry $G_2$,  the curvatures must satisfy:
\be(9\kappa-\lambda)(\kappa-9\lambda)\lambda=0.\label{waru}\ee
\end{theorem}

\begin{proof}
According to the previous discussion ${\mathcal D}_v$ will have local symmetry $G_2$ if and only if all $A_i$s given by the equations (\ref{eqg2}) identically vanish. This means that the following three equations are necessary and sufficient:
\be\begin{aligned}
&10(\kappa-\lambda)^3\kappa_{1111}-70(\kappa-\lambda)^2\kappa_{111}\kappa_1-49(\kappa-\lambda)^2\kappa_{11}^2+280(\kappa-\lambda)\kappa_1^2\kappa_{11}+\\&8(\kappa-\lambda)^3(2\kappa+7\lambda)\kappa_{11}-20(\kappa-\lambda)^2(\kappa+6\lambda)\kappa_1^2-175\kappa_1^4+\\&(\kappa-\lambda)^4(\kappa-9\lambda)(9\kappa-\lambda)=0,\\&\\
&-
10(\kappa-\la)^3a_2\kappa_{111}+\tfrac{154}{3}(\kappa-\lambda)^2a_2\kappa_{11}\kappa_1-20(\kappa-\lambda)^3a_2^2\kappa_{11}-\\&\tfrac43(\kappa-\lambda)^3(3\kappa-7\lambda)\kappa_{11}-\tfrac{140}{3}(\kappa-\lambda)a_2\kappa_1^3+\tfrac53(\kappa-\lambda)^2(21a_2^2+4\kappa-11\lambda)\kappa_1^2-\\&\tfrac43(\kappa-\lambda)^3(15a_2^2+12\kappa+7\lambda)a_2\kappa_1+\tfrac13(\kappa-\lambda)^4(\kappa-9\lambda)(9\kappa-\lambda)=0,\\&\\
&30(\kappa-\lambda)^3a_2^2\kappa_{11}-49(\kappa-\lambda)^2a_2^2\kappa_1^2+\\&2(\kappa-\lambda)^3(15a_2^2-3\kappa-28\lambda)a_2\kappa_1+(\kappa-\lambda)^4(\kappa-9\lambda)(9\kappa-\lambda)=0.
\end{aligned}\label{eqg22}\ee
One can view these equations as \emph{algebraic} equations on $\kappa_{1111}$, $\kappa_{111}$ and $\kappa_{11}$, and as such they may be easily solved. However, the solutions $\kappa_{1111}=\kappa_{1111}(\kappa,\kappa_1,\lambda,a_2)$,  $\kappa_{111}=\kappa_{111}(\kappa,\kappa_1,\lambda,a_2)$,  $\kappa_{11}=\kappa_{11}(\kappa,\kappa_1,\lambda,a_2)$, have to satisfy equations
$$\der \kappa_{11}=\kappa_{111}\sigma^1\quad{\rm and}\quad\der\kappa_{111}=\kappa_{1111}\sigma^1.$$
This introduces two additional algebraic equations involving the four variables $\kappa,\kappa_1,\lambda$ and $a_2$. Elimination of $\kappa_1$ from these two equations reduces them to a single algebraic equation for $\kappa,\lambda$ and $a_2$, which after a simplification, and exclusion of the possibility in which $a_2\equiv 0$ yields the necessary condition $(9\kappa-\lambda)(\kappa-9\lambda)\lambda=0$. The case $a_2\equiv 0$ must be excluded because otherwise, $\kappa\equiv 0$, and the equations (\ref{eqg22}) reduce to $\kappa=\lambda=0$. 
\end{proof}
We now have the corollary confirming the result of Zelenko-Bryant-Bor-Mont\-gom\-ery:
\begin{corollary}
Two surfaces of constant Gaussian curvature rolling on each other `without slipping or twisting' have ${\mathcal D}_v$ with local symmetry $G_2$ if and only if the ratio of their curvatures is $1:9$ or $9:1$  
\end{corollary}
\begin{proof}
Obviously $\kappa=9\lambda$ and $\kappa=\tfrac19\lambda$ solve (\ref{waru}). But instead of looking into the integrability conditions it is better now to use (\ref{eqg2}) to write down the Cartan quartic in this case. Of course now, since $\kappa={\rm const}$, we use equations (\ref{eqg2}) with $\kappa_{1111}=\kappa_{111}=\kappa_{11}=\kappa_1=0$. Inserting this into (\ref{eqg2}), and using the definition (\ref{cq}) of the Cartan quartic we find that, modulo a nonvanishing factor, the Cartan quartic is: 
\be C(\zeta)=(\kappa-9\lambda)(9\kappa-\lambda)(\kappa-\lambda)^4(1+2\zeta+2\zeta^2)^2.\label{qc}\ee
Excluding the integrable case, this is identically zero if and only if $\kappa=9\lambda$ or $\kappa=\tfrac19\lambda$, as claimed.
\end{proof}
\begin{remark}
Note that if $\lambda>0$ the surfaces described by the corollary are spheres with the ratio of the radii $1:3$ or $3:1$. But we can also have two hyperboloids with $\lambda<0$ here.
\end{remark}
\begin{remark}
Also note that if $(\kappa-9\lambda)(9\kappa-\lambda)\neq 0$ the Cartan quartic has always \emph{two distinct double} roots. In the terminology of Ref. \cite{franco}, the root type of the Cartan quartic is $[2,2]$. According to Cartan, in such a case, the dimension of the local symmetry group of the corresponding distribution $\mathcal D$ can not be larger than 6. It is easy to think of examples where symmetry group of ${\mathcal D}_v$ has the maximal dimension 6, since the distribution of the two-sphere system has always symmetry $\sog(3)\times\sog(3)$, and the distribution of the two-hyperboloid system has always symmetry $\sog(1,2)\times\sog(1,2)$. 
\end{remark}
\subsection{$G_2$ and surfaces of revolution rolling on the plane}\label{revo}
We now pass to the analysis of the still \emph{open} possibility $\lambda=0$ in (\ref{waru}). It turns out that in this case we
can obtain several examples of surfaces that roll `without slipping or twisting' on the \emph{plane} with velocity space ${\mathcal D}_v$ that has symmetry $G_2$.   

We have the following theorem.
\begin{theorem}\label{surfa}
Modulo homotheties there are only three metrics corresponding to surfaces with a Killing vector, which when rolling on the plane $\bbR^2$ `without slipping or twisting', have the distribution ${\mathcal D}_v$ with local symmetry $G_2$. 
These metrics in a convenient coordinate system can be written as
\be\boxed{\begin{aligned}
g_{1o}=&\rho^4\der\rho^2+\rho^2\der\varphi^2,\\
g_{1+}=&(\rho^2+1)^2\der \rho^2+\rho^2\der\varphi^2,\\
g_{1-}=& (\rho^2-1)^2\der \rho^2+\rho^2\der\varphi^2,
\end{aligned}}\label{three}
\ee
or, collectively as:
$$\boxed{
g_1=(\rho^2+\epsilon)^2\der \rho^2+\rho^2\der\varphi^2,\quad{\rm where}\quad\epsilon=0,\pm1.}$$
Their curvature is given by
\be
\kappa=\frac{2}{(\rho^2+\epsilon)^3}.\label{gauss}\ee   
\end{theorem}
\begin{proof}
If $\lambda=0$ and $(\Sigma_1,g_1)$ is a surface with a Killing vector $K=\partial_\varphi$, we can introduce local coordinate systems $(x,y)$ on $\Sigma_1$ and $(u,v)$ on $\Sigma_2$ such that
$$\sigma^1=\rho(x) \der x,\quad\sigma^2=\rho(x)\der \varphi,\quad \sigma^3=\der u,\quad\sigma^4=\der v.$$
Then the 4-metric reads:
$$g=\rho(x)^2(\der x^2+\der \varphi^2)-\der u^2-\der v^2,$$
and the variables from the equations (\ref{eqg22}) needed for the $G_2$ symmetry are given by:
$$a_2=-\frac{\rho'}{\rho^2},\quad\kappa=\frac{{\rho'}^2-\rho''\rho}{\rho^4}.$$
The only relevant coframe derivative is given by:
$$\partial_1=\frac{1}{\rho}\partial_x.$$
One can now write down the equations (\ref{eqg22}) in this setting. They look ugly, and they all involve the derivatives of the function $\rho$ up to the sixth order. We treated these equations as algebraic equations for $\rho^{(6)}$, $\rho^{(5)}$ and $\rho^{(4)}$, and used the same trick as in the proof of Theorem \ref{consts}. Namely, we algebraically solved equation for $\rho^{(4)}$, differentiated it, and compared it with the $\rho^{(5)}$ obtained algebraically. Then we did the same for $\rho^{(5)}$ and $\rho^{(6)}$. This produced a unique compatibility condition, obviously of the third order in the derivatives of $\rho$, which reads:
\be
\rho^{(3)}\rho'\rho^2-3{\rho''}^2\rho^2+\rho''{\rho'}^2\rho+{\rho'}^4=0.\label{adam}\ee
Now two miracles has happened: It turns out that 
\begin{itemize}
\item the equation (\ref{adam}) is not only necessary but also sufficient for making the Cartan quartic vanishing, and also
\item the equation (\ref{adam}), despite its ugly look, is completely solvable by means of elementary functions\footnote{We thank Adam Szereszewski \cite{aszer} for showing us the explicit transformation from equation (\ref{adam}) to (\ref{ada1}).}.
\end{itemize}
The first claim can be easily checked by solving (\ref{adam}) algebraically for $\rho^{(3)}$ and inserting it, together with its three consecutive derivatives into equations (\ref{eqg22}). These, with such $\rho^{(3)}$, become identities $0=0$. The second claim is justified by making a reciprocity transformation for the variables $x$ and $\rho$. It is an elementary calculation, that the function $\rho=\rho(x)\neq 0$ satisfies equation (\ref{adam}) if and only if $x=x(\rho)$ satisfies a \emph{linear} 3rd order ODE:
\be
x'''\rho^2+x''\rho-x'=0.\label{ada1}\ee
This can be easily solved yielding
$$x=\tfrac12 \alpha \rho^2+\beta \log \rho+\gamma$$
as its most general solution. Here $\alpha,\beta,\gamma$ are real constants. Inserting this general solution into the metric of the surface $\Sigma_1$ we get
$$g_1=\rho^2 (\der (\alpha \rho^2+\beta \log \rho+\gamma))^2+\rho^2\der \varphi^2=(\beta+\alpha \rho^2)^2\der \rho^2+\rho^2\der \varphi^2.$$
We have to exclude here the case when $\alpha=0$, since in this case $g_1$ is flat. If $\alpha\neq 0$ 
metrics $g_1$ are homothetic to one of the metrics (\ref{three}). 
In particular, all metrics $g_1$ with $\beta=0$ are homothetic to $g_{1o}$. If $\beta\neq0$ all metrics $g_1$ for which $\alpha\beta>0$ are homothetic to $g_{1+}$, and if $\alpha\beta<0$ the metrics $g_1$ are homothetic to $g_{1-}$. 

Calculating the Gauss curvature for the metrics $g_1$ above we get
$$\kappa=\frac{2\alpha}{(\beta+\alpha\rho^2)^3},$$
which reduces to (\ref{gauss}) for the three homothety non-equivalent classes of metrics $g_{1+}$, $g_{1-}$ and $g_{1o}$ . 
\end{proof}
Now the problem of isometric embedding of metrics (\ref{three}) in flat 
$\bbR^3$ arises. We have the following theorem.
\begin{theorem}\label{emb}
Let $\mathcal U$ be a region of one of Riemann surfaces $(\Sigma_1,g_1)$ of Theorem \ref{surfa}, in which the curvature $\kappa$ is nonnegative. In the case $\epsilon=+1$, such a region can be isometrically embedded in flat $\bbR^3$ as a surface of revolution. The embedded surface, when written in the Cartesian coordinates $(X,Y,Z)$ in $\bbR^3$, is algebraic, with the embedding given by
$$\boxed{(X^2+Y^2+2)^3-9Z^2=0, \quad\quad \epsilon=+1.}$$
In the case $\epsilon=-1$, one can find an isometric embedding in $\bbR^3$ of a portion of $\mathcal U$ given by $\varphi\in[0,2\pi[$, $\rho\geq\sqrt{2}$. This embedding gives another surface of revolution which is also algebraic, and in the Cartesian coordinates $(X,Y,Z)$, given by 
$$\boxed{(X^2+Y^2-2)^3-9Z^2=0,\quad\quad\epsilon=-1.}
$$
In the case $\epsilon=0$, one can embed a portion of $\mathcal U$ with $\rho\geq 1$ in $\bbR^3$ as a surface of revolution 
$$\boxed{Z=f(\sqrt{X^2+Y^2}),\quad{\rm with}\quad f(t)=\int_{\rho=1}^t \sqrt{\rho^4-1}~\der\rho.}
$$ 
\end{theorem}
\begin{figure}[htb]
\begin{center}
\includegraphics[scale=0.45]{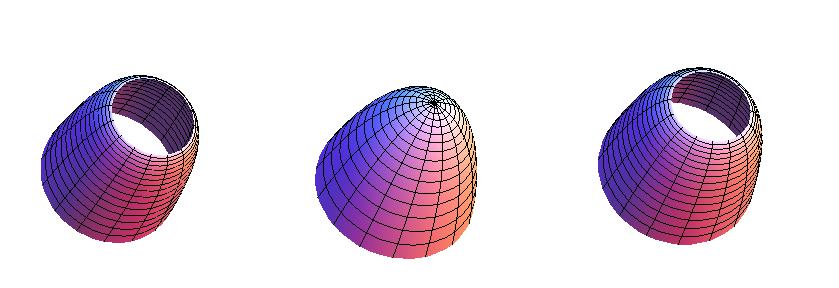}
\caption{The Mathematica print of the three surfaces of revolution, whose induced metric from $\bbR^3$ is given, from left to right, by respective metrics $g_{1-}$, $g_{1+}$ and $g_{1o}$. The middle figure embeds all $(\Sigma_1,g_{1+})$. In the left figure only the portion of $(\Sigma_1,g_{1-})$ with \emph{positive} curvature is embedded, and in the right figure only points of $(\Sigma_1,g_{1o})$ with $\rho>1$ are embedded. It is why the left and right figures have holes on the top. All three surface, when rolling on a plane `without twisting or slipping' have velocity space ${\mathcal D}_v$ with symmetry $G_2$.}
\end{center}
\end{figure}
\begin{proof}
In Theorem \ref{surfa}  we have proven that we have three cases of metrics corresponding to surfaces which when rolling `without slipping or twisting' on the plane have $G_2$ as a symmetry of ${\mathcal D}_v$. To embed them in $\bbR^3$ as surfaces of revolution, we put:
\be \boxed{X=\rho \cos\varphi,\quad Y=\rho\sin\varphi,\quad Z=\int \sqrt{(\rho^2+\epsilon)^2-1}~\der\rho, \quad{\rm where}\quad\epsilon=0,\pm1.}\label{mysurf}\ee  
If $\epsilon=1$, the function under the square root in the integral is positive for all $\rho\geq 0$, and this embeds the entire $(\Sigma_1,g_{1+})$ in $\bbR^3$. Actually the integral is elementary in this case, $\int \sqrt{(\rho^2+1)^2-1}~\der\rho=\tfrac13(\rho^2+2)^{3/2}$, and this embeds $(\Sigma_1,g_{1+})$ as an algebraic surface
$$(X^2+Y^2+2)^3-9Z^2=0$$
in $\bbR^3$. 
If $\epsilon=-1$ the integral for $Z$ above is only meaningful for $\rho\geq\sqrt{2}$. So only the region $\varphi\in[0,2\pi[$, $\rho\geq\sqrt{2}$ of $\Sigma_1$ can now be embedded in this way. In this case the integral for $Z$ is also elementary, $\int \sqrt{(\rho^2-1)^2-1}~\der\rho=\tfrac13(\rho^2-2)^{3/2}$. This embeds the above mentioned portion of $(\Sigma_1,g_{1-})$ as an algebraic surface (of revolution) in flat $\bbR^3$ given by:
$$(X^2+Y^2-2)^3-9Z^2=0.$$
If $\epsilon=0$ the embedding is given by (\ref{mysurf}) with $Z$ being an elliptic integral $Z=\int \sqrt{\rho^4-1}~\der\rho$. It now embeds the portion $\varphi\in[0,2\pi[$, $\rho>1$ of $(\Sigma_1,g_{1o})$ in flat $\bbR^3$. 
\end{proof}

The three surfaces of revolution defined in the Theorem \ref{emb}, and depicted in Figure 1 are examples of surfaces which when rolling on the plane `without slipping or twisting' have the velocity space ${\mathcal D}_v$ with $G_2$ as the group of local symmetries. Other isometric embeddings of these surfaces in $\bbR^3$ may provide other examples. For example it would be very instructive to find an isometric embedding in $\bbR^3$ of the positive curvature region $0<\rho\leq 1$, $0\leq\varphi<2\pi$ of the metric $g_{1o}=\rho^4\der\rho^2+\rho^2\der\varphi^2$, or of the positive curvature region $1<\rho<\sqrt{2}$ of the metric $g_{1-}=(\rho^2-1)^2\der\rho^2+\rho^2\der\varphi^2$. 
\begin{remark}
Although the region of negative curvature of the Riemann surface $(\Sigma_1,g_{1-})$ 
\begin{figure}[htb]
\begin{center}
\includegraphics[scale=0.45]{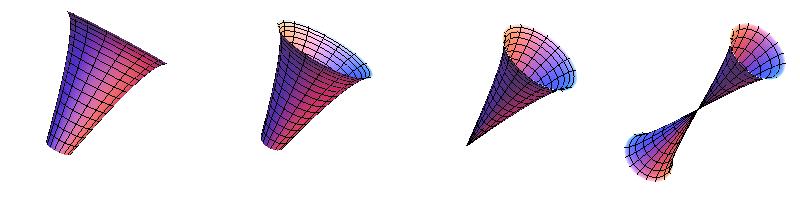}
\caption{The first two figures on the left give the Mathematica print of negative curvature portion $\varphi\in[0,2\pi[$, $\rho\in[0.5,2]$ of the Riemann surface with the metric $g_1=(\rho^2-5)^2\der\rho^2+\rho^2\der\varphi^2$ embedded as a surface of revolution. The metric $g_1\oplus(-\der u^2-\der v^2)$ has twistor distribution $\mathcal D$ with $G_2$ symmetry. The third figure is the print of the embedding for the $\rho$ range: $\rho\in[0,2]$. One sees that the embedding degenerates at $\rho=0$. The last picture gives the embedding for the maximally extend $\rho$ range, $\rho\in[-2,2]$.}
\end{center}
\end{figure}
can not be simply embedded as a surface of revolution in $\bbR^3$, we can find such an embedding for a portion of negative curvature region with a metric $g_1=(\rho^2-5)^2\der\rho^2+\rho^2\der\varphi^2$, which is homothetically equivalent to $g_{1-}$. As it is evident from the proof of Theorem \ref{surfa} this metric is in the class of metrics $g_1$, which together with the flat metric $g_2=\der u^2+\der v^2$, form the split signature metric $g=g_1\oplus(-g_2)$ with twistor distribution having symmetry $G_2$.  For this metric we consider a negative curvature region given by $\varphi\in[0,2\pi[$, $\rho\in]0,2[$, and we isometrically embed it in $\bbR^3$ via 
$$X=\rho\cos\varphi,\quad Y=\rho\sin\varphi,\quad Z=\int_0^\rho\sqrt{(x^2-6)(x^2-4)}\der x.$$ 
The embedding is not well defined in $\rho=0$, and the integral is not elementary. Nevertheless we can plot the obtained surface of revolution, which is depicted in Figure 2.
\end{remark}
\section{Acknowledgments}
This paper would never has been initiated if C. Denson (Denny) Hill, invited the second author (P. N.) to Stony Brook. Simply: without Denny's invitation we would never have met.

The paper owes much to the twistorial ideas of Roger Penrose, which were transmitted to us during many lectures of Andrzej Trautman at Physics Department of Warsaw University. In these lectures A. T. 
emphasized the role of the totally null planes in 4-dimensional geometry. Several times in his lectures, he underlined the differences of the geometry of totally null planes induced by metrics of different signatures. While learning about the main difference between the Riemannian or Lorentzian signature and the split signature, the split signature case became interesting for us. We took this difference as the main theme of this paper due to Andrzej Trautman. 

We also gratefully acknowledge that this work would never have been completed unless P. N. was invited for a lecture at the Erwin Schr\"odinger Institute, in connection with the conference entitled ``The interaction of geometry and representation theory. Exploring new frontiers'' organized by A. Cap, A. L. Carey, A. R. Gover, C. R. Graham and J. Slovak in Vienna, Austria. During this conference, which was devoted to Mike Eastwood's $60^{\rm th}$ birthday, Mike Eastwood answered our question on how to identify the spaces $C(\Sigma_1,\Sigma_2)$ and $\bbT(\Sigma_1\times\Sigma_2)$ by means of the graph of the map $A_\phi:\bbR^2\to\bbR^2$. The identification was intuitively obvious for us, but we were lacking the proper wording that Mike gave us.

Also, during (and after) this conference, Robert Bryant, gave us many suggestions which were very useful in the final stage of writing the paper. In particular, he gave us many comments regarding Sections \ref{killi} and \ref{revo} of the paper. One of them was to extend our Theorem \ref{surfa} to Theorem \ref{consts}. 

Helpful hints and comments from Thomas Leistner and Travis Willse are also gratefully acknowledged.
 
\end{document}